\documentclass[11pt,reqno]{amsart}
\usepackage{graphicx, color}
\usepackage{amssymb, amsmath,amsthm}
 
\newtheorem{theorem}{{\bf Theorem}}
\newtheorem{lemma}[theorem]{{\bf Lemma}}
\newtheorem{corollary}[theorem]{{\bf Corollary}}
\newtheorem{proposition}[theorem]{{\bf Proposition}}

\newtheorem{remark}[theorem]{Remark}

\numberwithin{equation}{section}

\begin{document}

%
%

\title[Normalized Entropy versus Volume]
{Normalized entropy versus volume \\
for pseudo-Anosovs}

\author[S. Kojima]{%
    Sadayoshi Kojima
}
\address{%
	Department of Mathematical and Computing Sciences \\
	Tokyo Institute of Technology \\
	Ohokayama, Meguro \\
	Tokyo 152-8552, Japan
}
\email{%
        sadayosi@is.titech.ac.jp
}
\author[G. McShane]{%
	Greg McShane 
} 
\address{%
	UFR de Math\'ematiques \\
	Institut Fourier 100 rue des maths \\
	BP 74, 38402 St Martin d'H\`eres cedex, France
}
\email{%
	Greg.McShane@ujf-grenoble.fr
}


\thanks{%
The first author is partially supported by Grant-in-Aid for Scientific Research (A)
 (No. 18204004), JSPS, Japan.
} 

\begin{abstract} 
Thanks to a recent result by Jean-Marc Schlenker, 
we establish an explicit linear inequality between the normalized entropies 
of pseudo-Anosov automorphisms and the hyperbolic volumes of 
their mapping tori.  
As its corollaries, 
we give an improved lower bound for values of entropies of pseudo-Anosovs 
on a surface with fixed topology, 
and a proof of a slightly weaker version of the result by Farb, Leininger and Margalit first, 
and by Agol later,   
on finiteness of cusped manifolds generating surface automorphisms 
with small normalized entropies.  
Also, 
we present an analogous linear inequality between 
the Weil-Petersson translation distance of 
a pseudo-Anosov map 
(normalized by multiplying the square root of the area of a surface)
and the volume of its mapping torus, 
which leads to a better bound.
\end{abstract}

\maketitle

%
%

\section{Introduction}\label{Sect:introduction}

Let  $\varSigma = \varSigma_{g,m}$ be an orientable surface of genus $g$ with $m$ punctures.
We will suppose that  $3g - 3 + m \geq 1$ so that
$\varSigma$  admits a Riemannian metric of constant curvature  $-1$,  
a hyperbolic structure of finite area, 
which, by Gauss-Bonnet, satisfies
 $\mathrm{Area}\, \varSigma = 2 \pi |\chi(\varSigma)| = 2\pi(2g-2+m)$  
with respect to the hyperbolic metric.  

The isotopy classes of orientation preserving automorphisms of  $\varSigma$, 
called mapping classes,  
were classified into three families
by Nielsen and Thurston  \cite{ThurstonII},
namely  periodic, reducible and pseudo-Anosov.  
Choose a representative  $h$  of a mapping class  $\varphi$, 
and consider its mapping torus, 
\begin{equation*} 
	\varSigma \times [0,1]/ (x, 1) \sim (h(x), 0).    
\end{equation*} 
Since the topology of the mapping torus depends 
only on the mapping class  $\varphi$,  
we denote its topological type by  $N_{\varphi}$  

A celebrated theorem by Thurston  \cite{ThurstonIII}  asserts that 
$N_{\varphi}$  admits a hyperbolic structure iff  $\varphi$  is pseudo-Anosov.  
By Mostow-Prasad rigidity 
a hyperbolic structure of finite volume in dimension 3 is unique 
and geometric invariants are in fact topological invariants.  
In \cite{KKT}, 
Kin, 
Takasawa and the first named author 
compared the hyperbolic volume of  $N_{\varphi}$, 
denoted by  $\mathrm{vol} \, N_{\varphi}$,   
with the entropy of  $\varphi$, 
denoted by  $\mathrm{ent} \, \varphi$.
By \textit{entropy} we mean 
the infimum of the  topological entropy of automorphisms 
isotopic to  $\varphi$.  
In particular, 
they proved 
that there is a constant  $C(g, m) > 0$  depending only on 
the topology of  $\varSigma$  such that 
\begin{equation*} 
	\mathrm{ent} \, \varphi \geq C(g, m) \, \mathrm{vol} N_{\varphi}.   
\end{equation*} 

This result only asserts the  existence of a constant  $C(g, m)$  
since the proof is based on a result of Brock \cite{Brock}  
involving several constants  for which, a priori,
it appears  difficult to compute sharp values. 
On the other hand, it is well known that the  infimum of
 $\mathrm{ent} \, \varphi /\mathrm{vol}\, N_{\varphi}$ is $0$.
In fact, Penner constructed examples   \cite{Penner}
which demonstrate that,  
as the  complexity  of  surface  increases,
the entropy of  a pseudo-Anosov can be arbitrarily close to $0$.
By J{\o}rgensen-Thurston Theory  \cite{ThurstonI},  
the infimum  of volumes of hyperbolic 3-manifolds is strictly positive
so $C(g, m)$  necessarily tends to   $0$  when  $g + m \to \infty$.

Our main theorem  gives an explicit value for   $C(g, m)$

\begin{theorem} \label{Thm:Main} 
The inequality,   
\begin{equation}\label{Eq:Main}  
	{\rm ent} \, \varphi  
	\geq 
	\frac{1}{3 \pi |\chi(\varSigma)|} \, {\rm vol} \, N_{\varphi},  
\end{equation} 
or equivalently,  
\begin{equation}\label{Eq:MainII}  
	2 \pi |\chi(\varSigma)| \, {\rm ent} \, \varphi  
	\geq 
	\frac{2}{3} \, {\rm vol} \, N_{\varphi} 
\end{equation} 
holds for any pseudo-Anosov  $\varphi$.  
\end{theorem} 

The  quantity appearing on the left hand side of  (\ref{Eq:MainII}) 
is often referred to  as the  {\it normalized entropy}.  
The main theorem can thus be restated informally:
``the normalized entropy over the volume is bounded from below 
by a psoitive  constant which does not depend on the topology of  $\varSigma$".

The value of  $C(g,m)$  above does not seem to be 
quite far from the sharp constant.  
For example, 
choose the case that the surface is the punctured torus,
so that  $g = 1,\,m  = 1$  and  $|\chi(\varSigma_{1,1})| = 1$.   
Then the inequality  (\ref{Eq:Main})  becomes 
\begin{equation*} 
	\frac{{\rm ent} \, \varphi}{{\rm vol} \, N_{\varphi}} 
	\geq  \frac{1}{3 \pi} = 0.10610 \dots 
\end{equation*} 
In this particular case, it is conjectured
(see Conjecture 6.10 in \cite{KKT}  with supporting evidence) 
that 
\begin{equation*} 
	\frac{{\rm ent} \, \varphi}{{\rm vol} \, N_{\varphi}} 
	\geq  \frac{\log \frac{3+\sqrt{5}}{2}}{2v_3} 
	= 0.47412 \dots, 
\end{equation*} 
where  $v_3 = 1.01494 \dots$  is the volume of 
the hyperbolic regular ideal simplex.  
The conjectured constant above is known to be attained by 
the figure eight knot complement 
which admits a unique  $\varSigma_{1,1}$-fibration.  

The first application of Theorem \ref{Thm:Main}, 
is an amelioration of the  lower bound,
\begin{equation*} 
	\mathrm{ent} \, \varphi
	\geq 
	\frac{\log 2}{4(3g-3+m)},
\end{equation*} 	
for the entropy of pseudo-Anosovs
on a surface due to Penner\cite{Penner},   
provided that there is at least one puncture.    

\begin{corollary}\label{Cor:EntropyEstimate} 
Let  $\varphi$  be a pseudo-Anosov on  $\varSigma_{g, m}$  
with  $m \geq 1$.   
Then 
\begin{equation*} 
	\mathrm{ent} \, \varphi 	
	\geq 
	\frac{2v_3}{3\pi |\chi(\varSigma)|} 
	= \frac{2v_3}{3\pi(2g-2+m)}.    
\end{equation*} 
\end{corollary}  

\begin{proof} 
It is known by Cao and Meryerhoff in  \cite{CM}  that 
the smallest volume of an orientable noncompact hyperbolic 
3-manifold is attained by the figure eight knot complement  
and it is  $2v_3$.  
Thus replacing  ${\rm vol} \, N_{\varphi}$  in (\ref{Eq:Main})  by  $2 v_3$, 
we obtain the estimate.   
\end{proof} 

If a manifold admits a fibration over the circle, 
its first Betti number is necessarily positive.  
It is conjectured that the smallest volume of a hyperbolic $3$-manifold 
with positive first Betti number is also  $2v_3$.  
If it were true, 
then we could drop the assumption on the number of punctures  `` $m \geq 1$"  in 
Corollary \ref{Cor:EntropyEstimate}.    

The second application of our main theorem  
is the proof of a slightly weaker form of 
Farb, Leininger and Margalit's 
finiteness theorem for small dilitation pseudo Anosovs.

\begin{corollary}[Farb, Leininger and Margalit \cite{FLM}, Agol \cite{Agol}] 
For any  $C>0$,  
there are finitely many cusped hyperbolic $3$-manifolds 
$M_k$
such that 
any pseudo-Anosov  $\varphi$  on  $\varSigma$  with 
$|\chi(\varSigma)| \, {\rm ent} \, \varphi < C$  
can be realized as 
the monodromy of a fibration 
on a manifold  obtained from one of the $M_k$
by an appropriate Dehn filling.  
\end{corollary} 

We note that Farb, Leininger and Margalit are able to obtain 
in addition that the $M_k$ are in fact fibered and the 
surgeries (fillings) are along suspensions of punctures  of the fiber.

\begin{proof} 
If  $|\chi(\varSigma)| \, {\rm end}\, \varphi$  is bounded 
from above by a constant  $C$, 
then it certainly bounds the volume of  $N_{\varphi}$  
by Theorem \ref{Thm:Main}.
Recall that the thin part of $N_{\varphi}$  
consists of neighborhoods of (rank 2) cusps 
and Margulis tubes around short geodesics
all pairwise disjoint.
Thus the boundary of the \textit{thick part},
that is  the complement of the thin part,
consists of finitely many tori
and $N_\varphi$ is obtained 
from the thick part by Dehn filling these.
The volume of  $N_{\varphi}$ bounds the volume of the 
thick part and,
 using a covering by (finitely many) metric balls, 
Jorgensen and Thurston \cite{ThurstonI}
have shown that there are only  finitely many possibilities 
for its  topological type.
\end{proof}

Finally, replacing the entropy by Weil-Petersson translation distance 
in the proof of Theorem \ref{Thm:Main}, 
yields  an explicit value for the constant 
appearing in the upper bound
for volume of the mapping torus in Brock's Theorem 1.1
 \cite{Brock}.

\begin{corollary} \label{Cor:WP} 
If  $\varSigma$  is compact, 
then, 
the inequality,   
\begin{equation*}
	|| \varphi ||_{{\rm WP}}  
	\geq 
	\frac{2}{3 \sqrt{2\pi |\chi(\varSigma)|}} \, {\rm vol} \, N_{\varphi},  
\end{equation*} 
or equivalently,  
\begin{equation*}
	\sqrt{2 \pi |\chi(\varSigma)|} \, || \varphi ||_{{\rm WP}}   
	\geq 
	\frac{2}{3} \, {\rm vol} \, N_{\varphi} 
\end{equation*} 
holds for any pseudo-Anosov  $\varphi$,  
where  $|| \cdot ||_{{\rm WP}}$  is the Weil-Petersson 
translation distance of  $\varphi$.  
\end{corollary} 

It was pointed out to the authors by McMullen that 
there is a family of pseudo-Anasov automorphisms $\varphi_k$
such that $|| \varphi_k ||_{{\rm WP}}$ are bounded whilst the entropy
of $\varphi_k$,
which is just the translation distance for the Teichm\"uller metric,
diverges.
In fact, one can construct such examples 
so that the mapping tori $N_{\varphi_k}$
converge to a cusped hyperbolic 3-manifold
(though this limit may not be a surface bundle).
This can be interpreted as showing 
that  the relationship between volume 
and Weil-Petersson  distance is stronger 
than that with the Teichm\"uller distance.
Indeed, Brock has shown that there a lower bound for volume
in terms of $|| \varphi ||_{\rm WP}$
but the method that we present does not as yet extend to prove this.

The proof of Theorem \ref{Thm:Main} has two main ingredients:
The first is  Krasnov and Schlenker's recent results  on the renormalized
volumes of quasi-Fuchsian manifolds  \cite{KS, Schlenker};  
The second is the work of 
McMullen  \cite{McMullen} and Brock-Bromberg \cite{BB}
on geometric inflexibility   to obtain 
convergence in Thurston's Double Limit Theorem.
In the next section we review  briefly the requisite results  of Krasnov and Schlenker
and obtain an intermediate inequality (Corollary \ref{Vol vs T})
between the volume of the  convex core  of a quasi-Fuchsian manifold
and Teichm\"uller distance.
In the proceeding section, 
we prove Theorem \ref{Thm:Main} and Corollary \ref{Cor:WP}.  
The last section is an appendix where 
we present a simplified exposition 
of the ideas behind geometric inflexibility and 
a proof of the convergence result stated in Brock \cite{Brock}  which 
leads to a proof of the main theorem.  
\medskip 

{\bf Acknowledgement:} 
The authors are indebted to 
Ian Agol, 
Martin Bridgemann, 
Jeff Brock, 
Ken Bromberg, 
Dan Margalit, 
Curt McMullen, 
Kasra Rafi,
Jean-Marc Schlenker 
and Juan Souto 
for their valuable suggestions,  comments and encouragement
without which  this  paper
might never have been completed.

%
%

\section{preliminaries}


\subsection{Differentials}

Let  $R$  be a Riemann surface and 
 $T^{1,0}R$  and  $T^{0,1}R$  denote  respectively
 the holomorphic 
and the anti-holomorophic parts of the complex 
cotangent bundle,  
a canonical bundle,   over  $R$  respectively.  

A quadratic differential  $q$  on  $R$  locally expressed by  $q(z)dz^2$  is 
a section of the line bundle  $(T^{1,0}R)^{\otimes 2}$.  
A Beltrami differential  $\mu$  on  $R$  locally expressed by  
$\mu(z) \frac{d \bar{z}}{dz}$  
is a section of the line bundle  $T^{0,1}R \otimes (T^{1,0})^*R$
They are main players in  Teichm\"uller theory.  
A Beltrami differential  $\mu$  can be interpreted as a representative of 
an infinitesimal deformation of complex structure on  $R$.  
Hence,  
to each $\mu$  there is a corresponding tangent vector to  Teichm\"uller space  
$\mathcal{T} = \mathcal{T}_{g, m}$  at  $R$, 
however,
there is  an  infinite dimensional subspace that represents the trivial deformation
 so that the correspondence is not injective.
We now describe a construction  which allows one to eliminate this ambiguity.

Let  $Q(R)$  be the space of {\it holomorphic} quadratic differentials.
By Riemann-Roch, 
the dimension of  $Q(R)$  is  $3g-3+m$, 
that is, equal to that of  $\mathcal{T}$.  
We define  the $L^1$-norm on $Q(R)$  by
\begin{equation*} 
	||q||_1 = \int_R |q|.
\end{equation*} 
We note that, since it is finite dimensional, all norms on $Q(R)$  are equivalent
and we shall compare the  $L^1$-norm to another norm, tne $L^\infty$-norm,
in Paragraph \ref{comparing norms}.

Given a Betrami differential $\mu$  and  a quadratic differential $q$ the product of  
$\mu$  and  $q$  is a section of the line bundle 
$T^{0,1}R \otimes T^{1,0}R$, 
and there is a natural pairing defined by 
\begin{equation*} 
	(q, \, \mu) = \int_R \mu q.  
\end{equation*} 
Let  $L_{\infty}(R)$  be the space of 
uniformly bounded Beltrami differentials with respect to the norm, 
\begin{equation*} 
	|| \mu ||_{\infty} = \sup_{\| q  \|_1 = 1} |(q, \, \mu)|, 
\end{equation*} 
namely,   
$L_{\infty}(R) = \{ \mu \, ; \, ||\mu||_{\infty} < \infty \}$.  
Define $K$ to be the subspace of $ L_{\infty}(R) $
\begin{equation*} 
	K = \{ \mu \, ; \, (q, \, \mu) = 0 \; \; \text{for all $q \in Q(R)$} \}.  
\end{equation*} 
Then  $L_{\infty}(R)/K$  can  be identified with the tangent space  
$T_R \mathcal{T}$  of the Teichm\"uller space  $\mathcal{T}$  at  $R$, 
and moreover 
the pairing induces an isomorphism of  $Q(R)^*$ 
with  $L_{\infty}(R)/K \cong T_R \mathcal{T}$.  
Thus, 
$Q(R)$  can be regarded as a cotangent space of  $\mathcal{T}$  at  $R$, 
$T^*_R \mathcal{T}$,    
and  $( \; , \; )$  induces the duality pairing, 
\begin{equation}\label{Eq:DualityPairing}  
	( \; , \; ) : Q(R) \times L_{\infty}(R)/K \to \mathbb{C}.  
\end{equation}


\subsection{Projective Structures} 

A projective structure on a surface  $\varSigma$  is 
a special type of complex structure being
locally modeled on the geometry of the complex projective line
$(\widehat{\mathbb{C}}, \, \mathrm{PSL}(2, \mathbb{C}))$, 
where  $\widehat{\mathbb{C}} = \mathbb{C} \cup \{ \infty \}$ 
is the  Riemann sphere.  
The projective structure on  $\varSigma$  
is given by  an atlas such that each transition map is
the restriction of some element in
$\mathrm{PSL}(2, \mathbb{C})$.
Clearly these transition maps are holomorphic 
and so there is a unique complex structure 
naturally associated to a given projective structure.
Let  $X$  be a surface homeomorphic to  $\varSigma$  together with 
a projective structure, 
and let  $R$  be its underlying Riemann surface. 
When  $\chi(\varSigma) < 0$ 
there is a bijection between projective structures 
and holomorphic quadratic differentials.
To see this, 
recall that by the Uniformization Theorem 
the universal cover of  $X$  can be 
identified with the Poincar\'e disk  $\mathbb{D} = \{ z \in \mathbb{C} \, ; \, |z| < 1 \}$.  
The developing map  $f : \widetilde{X} \to \widehat{\mathbb{C}}$  
can be  regarded as a meromorphic function on  $\mathbb{D}$.  
Now  the Schwarzian derivative  $S(f)$  of  $f$  defines 
a holomorphic quadratic differential  $q$  on  $R$.  
Conversely, 
given a holomorphic quadratic differential  $q$  on  $R$,  
the Schwarzian differential equation 
\begin{equation*} 
	S(f) = q 
\end{equation*} 
has a solution 
which gives rise to a developing map 
of some complex projective structure on  $R$.  
Thus, 
there is a one-to-one correspondence between 
the set of all complex projective structures on  $R$  and 
$Q(R)$.


\subsection{Quasi-Fuchsian Manifolds}  

A quasi-Fuchsian group is defined to be 
a discrete subgroup  $\Gamma$  of  ${\rm PSL}(2, \mathbb{C})$  
such that its limit set  $L_{\Gamma}$  of  $\Gamma$  on 
the boundary  $\partial \mathbb{H}^3 = S^2_{\infty}$  
is either a circle or a quasi-circle (an embedded copy of the circle  
with Hausdorff dimension strictly greater than 1).  
This definition implies many consequences.  
For example, 
the domain of discontinuity  
$\Omega_{\Gamma} = S^2_{\infty} - L_{\Gamma}$  consists of 
exactly two simply-connencted domains,   
denoted by  $\Omega_{\Gamma}^+$  and  $\Omega_{\Gamma}^-$,   
the quotients of  $\Omega_{\Gamma}^{\pm}$  by  $\Gamma$  
are the same 2-orbifold  $O$  but with opposite orientations,  
and  $\mathbb{H}^3/\Gamma$  is a geometrically finite 
hyperbolic 3-orbifold homeomorphic to  $O \times \mathbb{R}$.  
If  $\Gamma$  is torsion free, 
then  $O$  becomes a surface  $\varSigma$  
and, in particular, 
$\Gamma$  is isomorphic to the fundamental group of  $\varSigma$.
In this case we say that
the quotient of  $\mathbb{H}^3$  by the  quasi-Fuchsian group  $\Gamma$
is a \textit{quasi-Fuchsian manifold.}

The action of  $\Gamma$  on  $\Omega_{\Gamma}$  is holomorphic 
so  
$X = \Omega_{\Gamma}^{+}/\Gamma$  and 
$Y = \Omega_{\Gamma}^{-}/\Gamma$
are  marked Riemann surfaces.  
Thus, 
there is a well defined map taking
a quasi-Fuchsian manifold  $\mathbb{H}^3/\Gamma$ to 
a pair of marked Riemann surfaces  $X, \, Y$  in 
the Teichm\"uller space  $\mathcal{T}$  of  $\varSigma$.  
In  \cite{BersI}  Bers showed  
that this map has an inverse.
He obtains as a corollary  a parametrization of 
the set of quasi-Fuchsian manifolds  
by the $\mathcal{T} \times \overline{\mathcal{T}}$.  
In other words, 
for any  pair $(X,  Y) \in \mathcal{T} \times \overline{\mathcal{T}}$,  
there is a unique quasi-Fuchsian manifold  $QF(X, Y)$.  
As noted above, 
the limit set of a quasi-Fuchsian group is either a circle or a quasi-circle,
and the quotient of its convex hull in  $\mathbb{H}^3$  by  $\Gamma$, 
denoted $C(X, Y)$, is homemorphic to  
$\varSigma$  or  $\varSigma \times [0, \, 1]$  accordingly, 
called the \textit{convex core}.   
It is known to be the smallest convex subset   
homotopy equivalent to  $QF(X, Y)$.  

Since the action of a quasi-Fuchsian group  $\Gamma$ on  $\Omega_{\Gamma}$  is 
linear fractional,  
the Riemann surfaces  $X$ and $Y$  are equipped 
not only with complex structures
but also with complex projective structures.  
Thus we have associated holomorphic quadratic differentials   
$q_X$  and  $q_Y$.    
Let  $q$  denote the  unique holomorphic quadratic differential on  
$X \sqcup Y$  
such that its restriction to  $X$  is  $q_X$  
and to  $Y$  is  $q_Y$.  

The notations  $q_X, \, q_Y$  may be a bit misleading 
since they both could vary even if one of complex structures  
of  $X$  or  $Y$  stays constant.  
However, 
as long as discussing quasi-Fuchsian deformations, 
we regard  $q_X$  as a complex projective structure of  
the Riemann surface on the left and  $q_Y$  on the right.  
This convention would resolve confusion of notations.  

\vspace{.2in}
\begin{figure}[h]
\begin{center}
\includegraphics[scale=.4]{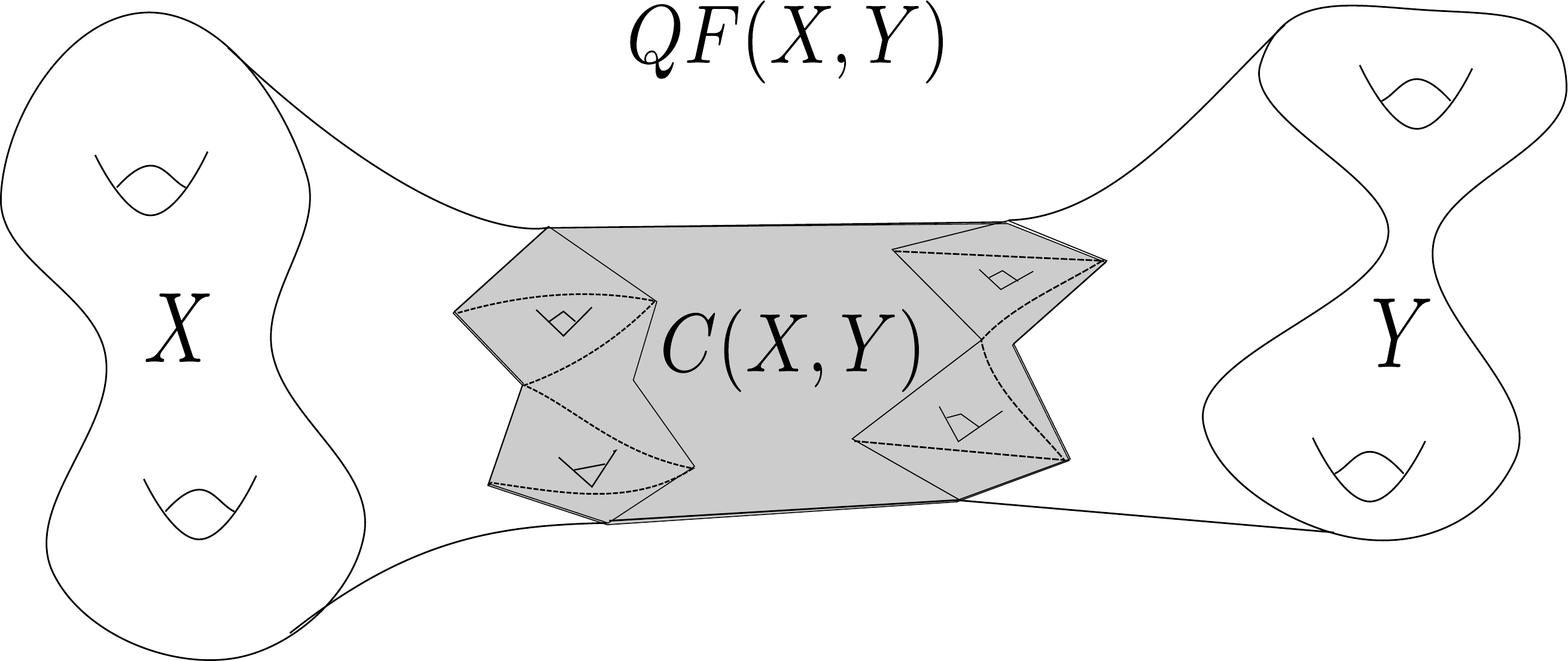} 
\end{center}
\caption{The quasi-Fuchsian manifold, the surfaces at infinity and the convex core.}
\end{figure}


\subsection{Renormalization of Volume}  

Renormalization of the volume of convex cocompact hyperbolic 3-manifolds 
were studied extensively by Krasnov and Schlenker in \cite{KS}.
In the following paragraphs we recall Krasnov and Schlenker's
results focusing on quasi-Fuchsian case.
Note that the surface at infinity  
 $\Omega_{\Gamma}/\Gamma$ has 2 connected components.

Throughout the rest of this section, 
we assume that  $\varSigma$  is compact.  
Let  $M$  be a quasi-Fuchsian manifold  $\mathbb{H}^3/\Gamma$  homeomorphic 
to  $\varSigma \times \mathbb{R}$.  
Following \cite{KS} we say that a
codimension-zero smooth compact convex 
submanifold  $N \subset M$  is \textit{strongly convex} if
the normal hyperbolic Gauss map 
from  $\partial N = \partial_+ N \sqcup \partial_- N$  
to the 
boundary at infinity  $\Omega_{\Gamma}/\Gamma$ is a homeomorphism.   
For example, 
a closed $\varepsilon$-neighborhood of the convex core of 
a quasi-Fuchsian manifold is strongly convex.   
Let $S_0 = \partial N$ 
then there is a family of  surfaces  $\{ S_r \}_{r \geq 0}$   
equidistant to $S_0$ foliating the ends of  $M$. 
If $g_r$ denotes the induced metric on   $S_r$,   
then define a metric at infinity associated 
to the family $\{ S_r \}_{r \geq 0}$  by
\begin{equation*} 
	g = \lim_{r \to \infty} 2 e^{-2r} \, g_r.  
\end{equation*}  
The resulting metric  $g$  in fact belongs to the \textit{conformal class at infinity}
that is the conformal structure 
determined by the complex structure on $\Omega_{\Gamma}/\Gamma$.  
It is easy to see that if we start with a strongly convex submanifold bounded 
by  $S_{r_0}$  for some  $r_0 > 0$, 
then the limiting metric is  $e^{2r_0}g$.  
Namely, 
if we shift the parametrization of an equidistant foliation by  $r_0$, 
then the limiting metric changes only by scaling  $e^{2r_0}$.  

Conversly,  
if $g$  is a Riemannian metric in the conformal class at infinity,  
then Theorem 5.8 in  \cite{KS}  shows that 
there is a unique foliation of the ends of $M$ by equidistant surfaces 
with compatible parametrization of leaves starting  $r_0 \geq 0$  
so that the associated metric at infinity is equal to  $g$.  
Notice that the parametrization may have to start with a positive  $r_0$.    
The construction of a foliation is due to Epstein  \cite{Epstein}.  
Then, 
a natural quantity to study in the context  
of strongly convex submanifolds  $N \subset M$ is  the \textit{W-volume} 
defined by 
\begin{equation}\label{Eq:W-volume}  
	 W(M, g):= \mathrm{vol} \, N_r - \frac{1}{4} \int_{S_r} H_r \, da_r + \pi r \chi(\partial M),  
\end{equation} 
where the parametrization  $r$  is induced by  $g$, 
$N_r$  is a strongly convex submanifold bounded by the associated leaf  $S_r$,   
$H_r$  is the mean curvature of  $S_r$  and  $da_r$  is the induced area form of  $S_r$.  
A simple computation
which can be found in \cite{Schlenker}  shows
that  the W-volume depends only on the metric at infinity  $g$,   
justifying the notation.  

The renormalized volume of  $M$  is now defined by  
\begin{equation*} 
	\mathrm{Rvol}(M) := \sup_{g} W(M, g),   
\end{equation*}  
where the supremum is taken over all metrics $g$ 
in the conformal class at infinity 
such that the area of each surface at infinity
$\Omega_{\Gamma}/\Gamma$ with respect to  $g$  is  $2\pi |\chi(\varSigma)|$.  
Section 7 in  \cite{KS}  presents an argument,
based on the variational formula stated 
in Corollary 6.2 in \cite{KS},  
that the supremum is in fact uniquely attained by 
the metric of constant curvature  $-1$.

We can now state, in a slightly modified form,
that one half of Theorem 1.1 in  \cite{Schlenker}  is as follows.  

\begin{theorem}[Theorem 1.1 in \cite{Schlenker} and its revised version]\label{Thm:Vol vs Rvol}  
Assume that  $\varSigma$  is compact.  
Then, 
there exists a constant  $D = D(\varSigma) > 0$  depending only on 
the topology of  $\varSigma$  such that 
the inequality,  
\begin{equation}\label{Eq:Vol vs Rvol}  
	\mathrm{vol} \, C(X, Y) \leq \mathrm{Rvol} \, QF(X, Y) + D,  
\end{equation} 
holds for any  $X, \, Y \in \mathcal{T}$.  
\end{theorem} 


\subsection{Variational Formula}  

In \cite{Schlenker} the metric at infinity of a quasi-Fuchsian manifold  $M$ 
which attains  the renormalized volume  is denoted  by  $\mathrm{I}^*$.  
In particular, 
the curvature of  $\mathrm{I}^*$  is constant  $-1$.  
The notation is consistent with the standard one for the first fundamental form.  
There is also an analogous notion  $\mathrm{II}^*$  of 
the second fundamental form on the surface at infinity with the same parametrization 
with appropriate scaling factor.  
More precisely, 
there is a unique bundle morphism  $B^*$  of the tangent space of boundary, 
corresponding to the shape operator,  
which is self-adjoint for  $\mathrm{I}^*$  and such that 
\begin{equation*} 
	\mathrm{II}^* = \mathrm{I}^*(B^*\cdot, \cdot) 
\end{equation*}  

The variational formula of the renormalized volume involves 
$\mathrm{II}^*, \, \mathrm{\dot{I}}^*$  and a Riemannian metric 
of constant curvature  $-1$  in the conformal class of the boundary.  
They all are symmetric $2$-tensors.  
In general, 
if we choose a local complex coordinate  $z = x + iy$,  
then 
we can express a symmetric $2$-tensor using the associated 
real coordinate  $(x, y)$  by  
$a (dx \otimes dx) + b (dx \otimes dy + dy \otimes dx) + c (dy \otimes dy)$  
and hence by a symmetric matrix 
\begin{equation*} 
	\begin{pmatrix} 
		a & b \\ 
		b & c 
	\end{pmatrix}.  
\end{equation*}  
Cororally 6.2 in \cite{KS}  states the variational formula of the W-volume 
as follows. 

\begin{lemma}[Proposition 3.10 \cite{Schlenker} and its revised version]\label{Lem:VariationalFormula}  
Under a first-order deformation of the hyperbolic structure on  $N$, 
\begin{equation*} 
	d \, {\rm Rvol} = - \frac{1}{4} \langle \mathrm{II}^*_0, \, \mathrm{\dot{I}}^* \rangle
\end{equation*} 
holds.  
Here  $\langle \; , \; \rangle$  is the extension to symmetric 2-tensors of 
the riemannian metric  $\rho^2 |dz|^2$  of constant curvature  $-1$  
defined by 
\begin{equation*} 
	\langle A, B \rangle = \int_R {\rm tr} \, (G^{-1}AG^{-1}B) \rho^2 dxdy,  
\end{equation*} 
where  $G$  is the metric tensor, 
and  $\mathrm{II}^*_0$  is a trace free part of  $\mathrm{II}^*$. 
\end{lemma} 

Krasnov and Schlenker found a remarkable relation
between $\mathrm{II}^*_0$,
the trace free part of  $\mathrm{II}^*$,
and the holomorophic  quadratic differential  $q$ 
corresponding to the projective structure of the boundary.  
To see this more precisely, 
recall that  $q$  has a local expression  $q = q(z) dz^2$  and 
$q(z) = f(z) + ih(z)$.  
Then  ${\rm Re} \, q = f dx \otimes dx - f dy \otimes dy - h (dx \otimes dy + dy \otimes dx)$  
as a symmetric 2-tensor can be expressed by a trace free symmetric matrix 
\begin{equation*} 
	\begin{pmatrix} 
		f & -h \\
		-h & -f 
	\end{pmatrix} 
\end{equation*}  
The identity below follows directly from  explicit formulae for 
the holomorphic quadratic differential  $q$  in question.  
Another more geometric proof can also be found  in the appendix of  \cite{KS}.  

\begin{lemma}[Lemma 8.3 in \cite{KS}]\label{Lem:II=QD} 
\begin{equation*} 
	\mathrm{II}^*_0 = - \mathrm{Re} \, q.  
\end{equation*}
\end{lemma} 

Notice that  $q = (f + ih)dz^2$  can be recovered from  $\mathrm{Re} \, q$.  
Passing through this identification and 
using an identification of  $Q(R)$  with  $T_R \mathcal{T}$  with respect to 
the metric  $\rho^2|dz|^2$, 
the variation of the metric 
\begin{equation*} 
	\dot{\mathrm{I}}^* = 
	\begin{pmatrix} 
		2\varphi & - 2\psi \\ 
		- 2\psi & - 2\varphi 
	\end{pmatrix} 
\end{equation*} 
can be transformed to the Beltrami differential 
\begin{equation*} 
	\mu = \frac{\varphi -i \psi}{\rho^2} \frac{d \bar{z}}{dz}.   
\end{equation*} 
This leads us to the reinterpration of the variational formula in Lemma \ref{Lem:VariationalFormula}  
in terms of the duality pairing  $( \; , \; )$  in (\ref{Eq:DualityPairing}).  

\begin{lemma}\label{Lem:Variation}  
Under a first-order deformation of the hyperbolic structure on  $N$, 
\begin{equation*} 
	d \, {\rm Rvol} = -  \mathrm{Re} \, (q, \, \mu) = - \int_R \mathrm{Re} \, \mu q 
\end{equation*} 
holds.  
\end{lemma}  
\begin{proof} 
Fix a local coordinate  $z = x + iy$  and let   $\rho^2 |dz|^2$ denote the hyperbolic metric then, 
\begin{align*} 
	\langle \mathrm{II}_0^*, \, \dot{\mathrm{I}}^* \rangle 
	& = \int_R \mathrm{tr} (G^{-1} \mathrm{II}_0^* G^{-1} \dot{\mathrm{I}}^*) \rho^2 \, dxdy \\
	& = 4 \int_R (f \varphi + h \psi) \rho^{-2} dxdy \\	
	& = 4 \int_R \mathrm{Re} \left(\frac{\varphi - i \psi}{\rho^2} (f + i h) \right) dxdy \\ 
	& = 4 \int_R \mathrm{Re} \, \mu q.  
\end{align*}  
\end{proof} 


\subsection{Rvol versus Teichm\"uller Distance}  \label{comparing norms}

We start with a quasi-Fuchsian manifold  $\mathbb{H}^3/\Gamma = QF(X, Y)$.  
Fix a conformal structure  $X$  on the left boundary component, 
and regard a conformal structure  $Y$  on the right as a variable.  
To each  $Y$, 
we assign an associated complex projective structure on  $X$  and 
therefore a holomorphic quadratic differential  $q_X$.  
This defines a map, 
\begin{equation*} 
	B_X : \mathcal{T} \to Q(X),   
\end{equation*} 
called a Bers embedding.  

Using the hyperbolic metric in the conformal class of  $X$, 
we can measure at each point of  $\varSigma$ 
the norm of  $q_X$.  
Let  $Q^{\infty}(X)$  be  $Q(X)$  
endowed with the  $L^{\infty}$  norm, 
namely, 
\begin{equation*} 
	|| q ||_{\infty} = \sup_{x \in X} \frac{|q(x)|}{\rho^2(x)},  
\end{equation*} 
where   $\rho |dz|$  defines the hyperbolic metric of 
constant curvature  $-1$.  

The following theorem with respect to the hyperbolic metric of 
constant curvature $-1$, 
due to  Nehari, can be found 
in a standard text book of the Teichm\"uller theory 
such as Theorem 1, p.134 in \cite{GL}.  

\begin{theorem}[Nehari \cite{Nehari}] 
The image of  $B_X$  in  $Q^{\infty}(X)$  is contained in the ball of radius  $3/2$.  
\end{theorem}  

Then, 
consider now the $L^1$-norm on  $Q(X)$   
and denote by  $Q^1(X)$  the vector space  $Q(X)$  
endowed with the  $L^1$  norm.  

\begin{corollary}\label{Cor:NehariL1} 
The image of  $B_X$  in  $Q^1(X)$  is contained in the ball of 
radius  $3 \pi |\chi(\varSigma)|$.  
\end{corollary} 
\begin{proof} 
The inequality 
\begin{equation*} 
	||q||_1 = \int_X |q| = \int_X \frac{|q|}{\rho^2} \, \rho^2 
	\leq ||q||_{\infty} \int_X \rho^2 = 2\pi |\chi(\varSigma)| \, ||q||_{\infty} 
\end{equation*} 
immediately implies the conclusion.  
\end{proof}  

The proof of the following comparison result is the same 
as that of Theorem 1.2 in  \cite{Schlenker}  by Schlenker with a different norm.  

\begin{proposition}\label{Prop:Rvol vs d_T}  
Suppose  $\varSigma$  is compact.  
The inequality,  
\begin{equation} \label{Rvol vs T}
	{\rm Rvol} \, QF(X, Y) \leq 3 \pi |\chi(\varSigma)|\, d_{\rm T}(X, Y),  
\end{equation} 
holds for any quasi-Fuchsian manifold  $QF(X, Y)$,  
where  $d_{\rm T}$  is the Teichm\"uller distance on  $\mathcal{T}$.  
\end{proposition} 
\begin{proof} 
Let  $Y : [0, d] \to \mathcal{T}$  be  the unit speed 
Teichm\"uller geodesic joining $X$  and  $Y$, 
so that, 
in particular,  
$Y(0) = X, \, Y(d) = Y$  and  $d = d_{\rm T}(X, Y)$.  
Then consider a one-parameter family of 
quasi-Fuchsian manifolds  $\{ QF(X, Y(t)) \}_{0 \leq t \leq d}$.   
By  Lemma \ref{Lem:Variation}  the variation
under the first-order deformation at time  $t$  is given by
\begin{equation*} 
	d \, {\rm Rvol} = 
	\mathrm{Re} 
	\left( (q_X(t), \dot{X})) + (q_{Y(t)}(t), \dot{Y}(t)) \right),  
\end{equation*} 
where  $\dot{X}, \, \dot{Y}(t)$  are tangent vectors of 
the deformation of complex structures 
on the two  ideal boundary components.  
Since  $\dot{X} = 0$,   
integrating the variation of ${\rm Rvol}$ along 
the path  $Y(t) \; (t \in [0,d])$, 
yields an expression for the renormalised volume
\begin{equation*} 
	{\rm Rvol} \, QF(X, Y) 
	= \mathrm{Re} \int_{t=0}^d (q_{Y(t)}(t), \dot{Y}(t)) dt.  
\end{equation*}   
On the other hand, 
\begin{align*} 
	|(q_{Y(t)}(t)), \dot{Y}(t))| 
	& = \left| \int_R q_{Y(t)}(t) \, \dot{Y}(t) \right| \\  
	& = \int_R |q_{Y(t)}(t) \, \dot{Y}(t)| \\  	
	& \leq ||q_{Y(t)}(t) ||_1 \, ||\dot{Y}(t)||_{\infty} 
\end{align*}  
where  $|| \cdot ||_{\infty}$  is the supremum norm on  $L_{\infty}(Y(t))$  
which is the dual to the  $L^1$-norm on  $Q(Y(t))$  and hence 
an infinitesimal form of the Teichm\"uller metric.  
By 
Corollary \ref{Cor:NehariL1},  for all  $t \in [0,d]$ one has , 
\begin{equation*} 
	||q_{Y(t)}(t) ||_1 \leq 3\pi |\chi(\varSigma)|,
\end{equation*} 
and,  by definition,  $|| \dot{Y}(t) ||_{\infty} = 1$   
so that the inequality now follows.
\end{proof} 

Replacing (\ref{Rvol vs T}) in  the inequality in Theorem \ref{Thm:Vol vs Rvol} one has,

\begin{corollary}\label{Vol vs T}
With the notation above,
\begin{eqnarray}
\mathrm{vol} \, C(X, Y) \leq 3 \pi |\chi(\varSigma)| \, d_{\rm T}(X, Y)  + D
\end{eqnarray}
\end{corollary}

%
%

\section{Proof} 

We first deal with the case that  $\varSigma$  is compact.  
Let  $\varphi$  be a pseudo-Anosov automorphism on  $\varSigma$  
and choose a marked Riemann surface  $X \in \mathcal{T}$  on 
the Teichm\"uller geodesic invariant by  $\varphi$.   
Remember that  $\varphi$  acts naturally on  $\mathcal{T}$  
by pre-composing  $\varphi^{-1}$  to the marking of  $X \in \mathcal{T}$  
and consider a family of quasi-Fuchsian manifolds   
$\{ QF(\varphi^{-n} X, \varphi^n X) \, ; \, n \in \mathbb{Z} \}$.  
These manifolds are quite close to
the infinite cyclic covering space of  $N_{\varphi}$  
if  $n$  is sufficiently large.   
Applying Corollary \ref{Vol vs T}
and dividing by  $2n$,  
we obtain the following estimate,  
\begin{equation}\label{Eq:Vol vs d_T}  
	\frac{1}{2n} {\rm vol} \, C(\varphi^{-n} X, \varphi^n X) 
	\leq \frac{1}{2n} \left( 3 \pi |\chi(\varSigma)| \, d_{\rm T}(\varphi^{-n} X, \varphi^n X) + D \right).  
\end{equation} 
We consider the limit as  $n \to \infty$
beginning with the right hand  side.

By a result of Bers in \cite{BersII}  (cf, \cite{Kojima}), 
we know that 
\begin{equation}\label{Eq:limit}  
	\lim_{n \to \infty} \frac{1}{2n} \, d_{\rm T} (\varphi^{-n}X, \varphi^nX) 
	= || \varphi ||_{\rm T} 
	= {\rm ent} \, \varphi
\end{equation} 
where  $|| \varphi ||_{\rm T}$  is the Teichm\"uller translation distance of  $\varphi$  
defined by 
\begin{equation*}
	|| \varphi ||_{\rm T} = \inf_{R \in \mathcal{T}} d_{\rm T}(R, \varphi R).  
\end{equation*} 
Since the constant  $D$  in  (\ref{Eq:Vol vs d_T})  does not depend on  $n$, 
the limit of the right hand  side is just a multiple of the   entropy. 

We now  consider the left hand side of (\ref{Eq:Vol vs d_T}).
Brock \cite{Brock} states  that  geometric inflexibility
should yield a proof that the limit as $n \rightarrow \infty$ 
exists and is equal to  ${\rm vol} \, N_{\varphi}$.  
However,  
at the time of writing,
it appears that there is no written proof of this fact.  
So,  
for completeness,  
we give its proof in the appendix, 
and we proceed the argument assuming this fact.

\begin{proof}[{\bf Proof of Theorem \ref{Thm:Main} for a compact surface $\varSigma$}]
The inequality in Theorem \ref{Thm:Main}  follows immediately from 
(\ref{Eq:Vol vs d_T}),  
(\ref{Eq:limit}) and 
the asymptotic behavior of the left hand side of 
(\ref{Eq:Vol vs d_T})  in  (\ref{Eq:vol N vs vol C}) 
for which we give a proof in the next section.  
\end{proof}

We now deal with the case that  $\varSigma$  has  $m$  punctures, 
where  $m \geq 1$.  
The argument reducing to the compact case below was 
suggested by Ian Agol.  
We begin with 

\begin{lemma} 
Suppose  $m \geq 2$,   
then there is an arbitrary high degree finite cover of  $\varSigma$  
so that the number of punctures is exactly $m$.
\end{lemma} 

\begin{proof}  
We construct the required  family of  coverings as follows.
Choose an increasing sequence of distinct primes  $\{p_i\}$  such that 
each  $p_i$  is coprime to  $m-1$.  
Then, 
there is a homomorphism  $\theta_i : \pi_1(\varSigma) \to \mathbb{Z}/p_i \mathbb{Z}$  
so that every element represented by a simple  loop around a puncture 
is mapped to a nontrivial element in  $\mathbb{Z}/p_i \mathbb{Z}$.  
The coverings associated to  $\{ {\rm Ker} \, \theta_i \}$  has the 
property in the statement.      
\end{proof} 

\begin{proof}[{\bf Proof of Theorem \ref{Thm:Main} for noncompact  $\varSigma$}]
Suppose  $m \geq 2$.
By the preceding lemma there is  a family of coverings  $\widetilde{\varSigma}_i \to \varSigma$  of 
increasing degrees  $d_i$ 
such that the number of punctures of each cover is just $m$.  
For each  $i$,  
there is  $k_i$  so that the action of  $\varphi^{k_i}$  leaves  
$\pi_1(\widetilde{\varSigma}_i)$  invariant in  $\pi_1(\varSigma)$
and so  $\varphi^{k_i}$  lifts to an automorphism  $\widetilde{\varphi}_i$  of 
$\widetilde{\varSigma}_i$.  
Filling  $m$  punctures of  $\widetilde{\varSigma}_i$,  
we obtain a compact surface  $\overline{\varSigma}_i$  and 
an induced automorphism of   $\overline{\varphi}_i$  of $\overline{\varSigma}_i$
which is pseudo-Anosov. 
The construction guarantees that we have the following relation between the entropies of 
the automorphisms,
\begin{equation*} 
	{\rm ent} \, \overline{\varphi}_i = {\rm ent} \, \varphi^{k_i} =  k_i \, {\rm ent} \, \varphi. 
\end{equation*} 
Applying the estimate for compact surfaces obtained above one obtains, 
\begin{equation*} 
	2\pi (d_i | \chi(\varSigma) | + m) \, {\rm ent} \, \overline{\varphi}_i 
	\geq \frac{2}{3} {\rm vol} \, N_{\overline{\varphi}_i}.  
\end{equation*} 
Now, dividing both sides by $k_id_i$  and letting  $i \to \infty$,
the limit of the left hand side is   $2\pi |\chi(\varSigma)| {\rm ent} \, \varphi$
whilst,  by Thurston's Orbifold Dehn Filling Theorem,
the limit of the right hand side is  $\frac{4}{3} {\rm vol} \, N_{\varphi}$.   
Thus we  obtain  the inequality  (\ref{Eq:MainII})  for non compact 
surfaces provided  $m \geq 2$. 

Finally, if  $m = 1$  then any finite abelian cover   
$\tilde{\varSigma} \rightarrow \varSigma$  of degree $p\geq 2$ 
has more than one puncture,
and there is  $k>0$  so that  $\varphi^k$  lifts to  $\tilde{\varSigma}$.
One verifies that the corresponding mapping torus is a  degree $kp$ 
cover of  $N_\varphi$ and the entropy of  $\varphi^k$ is $k\, {\mathrm ent}\, \varphi$.
Hence this case reduces  to the previous case  ($m \geq 2$).  
\end{proof}

It remains to prove Corollary  \ref{Cor:WP}  so 
consider the $L^2$ inner product on 
$Q(R) \cong T_R^* \mathcal{T}$  defined by 
\begin{equation*} 
	\langle q, q' \rangle = \int_R \frac{\bar{q}q'}{\rho^2}, 
\end{equation*}  
and recall that the Weil-Petersson metric  $d_{{\rm WP}}$  on the Teichm\"uller space  $\mathcal{T}$  is 
a Riemannian part of the dual Hermitian metric to the above co-metric 
on the cotangent space.  
Then, 
Weil-Peterson translation distance is defined by 
\begin{equation*}
	|| \varphi ||_{\rm WP} = \inf_{R \in \mathcal{T}} d_{\rm WP}(R, \varphi R).  
\end{equation*}

\begin{proof}[{\bf Proof of Corollary \ref{Cor:WP}}] 
If one starts the proof of Theorem \ref{Thm:Main}  for compact surfaces 
with Theorem 1.2 in \cite{Schlenker}  
which asserts  
\begin{equation*} 
	{\rm Rvol} \, QF(X, Y) \leq \frac{3 \sqrt{2\pi |\chi(\varSigma)|}}{2} \, d_{\rm WP}(X, Y),  
\end{equation*} 
instead of Proposition \ref{Prop:Rvol vs d_T}, 
and applies    
\begin{equation*} 
	|| q ||_2^2 
	= \int_R \frac{|q|^2}{\rho^2} 
	\leq ||q||_1 ||q||_{\infty} 
	\leq 2\pi |\chi(\varSigma)| \,  || q ||_{\infty}^2,    
\end{equation*} 
where  $|| \cdot ||_2$  is the $L^2$ norm induced by  $\langle  \cdot , \cdot \rangle$,   
we obtain the desired estimate.  
\end{proof} 

\begin{remark}  
If we prove {\em Corollary \ref{Cor:WP}}  first, 
then {\em Theorem \ref{Thm:Main}} for the compact case  follows immediately
by the  Cauchy-Schwarz inequality,
\begin{equation*} 
	|| q ||_1^2 
	= \left( \int_R |q| \right)^2 
	= \left( \int_R \rho \cdot \frac{|q|}{\rho} \right)^2 
	\leq \int_R \rho^2 \cdot \int_R \frac{\bar{q}q}{\rho^2} 
	= 2\pi |\chi(\varSigma)| \, ||q||_2^2,  
\end{equation*} 

\end{remark}

%
%

\bigskip 

\section{Appendix} 

We now consider the left hand side of  (\ref{Eq:Vol vs d_T}).  
Brock  \cite{Brock}  states that 
geometric inflexibility should yield a proof that the limit as  $n \to \infty$  
\begin{equation*} 
	{\rm vol} \, C(\varphi^{-n}X, \varphi^nX) = 2n \, {\rm vol} \, N_{\varphi} + O(1).   
\end{equation*} 
We give a proof of this using Minsky's work on bilipschitz models to simplify certain points.  

It is useful to think of  the convex hull $C_n$ of  $Q_n = QF(\varphi^{-n}X, \varphi^{n}X)$ 
as being ``modelled" on  another 3-manifold $M_n$ as follows.
The boundary $\partial C_n$ consists of a pair of surfaces  
$\partial^\pm  C_n$ each homeomorphic to $\varSigma$.
By the work of Epstein-Marden-Markovic \cite{EMM}  and  
Bridgeman \cite{Bridgeman}, 
these  surfaces (equipped with their path metrics)
 are   5-bilipschitz  to $\varphi^{\pm n} X$ 
 respectively (equipped with their Poincar\'e metrics
 and the markings $\varphi^{\mp n}$.)
Thus the boundary components of the convex core 
are ``modelled" on the surfaces at infinity
in that they have ``roughly the same geometry".
Ideally one would like to extend this equivalence
allowing us to think of  the convex hull $C_n$
as being ``modelled" on  a 3-manifold, which we denote by  $M_n$,
that is none other than  the portion of the
universal curve above the axis of $\varphi$
between the points $\varphi^{-n}X$ and $\varphi^{n}X$.
In fact, by a theorem of  Minsky \cite{Minsky} and additional work of Rafi \cite{Rafi}
the convex core $C_n$ is uniformly bilipschitz to $M_n$  equipped with a metric which we now describe.
We parameterise the Teichm\"uller geodesic  $g$  between $\varphi^{-n}X$ and $\varphi^{n}X$ 
by arclength and identify its source with $[-n \, {\rm ent} \, \varphi,  n \, {\rm ent} \, \varphi]$.
For $t \in [-n \, {\rm ent} \, \varphi,  n \, {\rm ent} \, \varphi]$
the metric on $\varSigma \times \{t\}$ is the metric assigned by  $g(t)$.
The distance between  $\varSigma \times \{t\}$ and  $\varSigma \times \{s\}$
is $| s - t |$.
There are three important consequences of the existence of Minsky's model $M_n$:
\begin{enumerate}
\item 
	Recall that a hyperbolic 3-manifold has \textit{bounded geometry} if the injectivity 
	radius is bounded below by a positive constant.
	All the manifolds that we consider have \textit{uniformly bounded geometry}.
	that is, the  injectivity radius of the sequence $Q_n$ is  bounded from below 
	by a  constant $\epsilon_0>0$.
	Under this hypothesis,  a version of the  Morse Lemma says that  a closed  curve 
$\gamma$ of length $l(\gamma)$ is contained in an $R$-regular neighborhood
of the closed geodesic in its homotopy class where 
$R = l(\gamma)/2 + \cosh^{-1} (l(\gamma)/(2\epsilon_0) )$.
	
\item 
	Let $\gamma_*$ be  a (short, simple)  closed curve on $X$
	and we denote the  geodesic representative of  $\varphi^k (\gamma_*)$  in  $C_n$  
	by  $\varphi^k(\gamma)$.
	Then the lengths of the geodesics
 	$\varphi^{k}(\gamma)$  
	for  $-n \leq k \leq n$  are uniformly bounded.
\item 
	If $-n \leq i, j \leq n$,  
	then the distance between geodesics 
	$\varphi^{i}(\gamma)$  and $\varphi^{j}(\gamma)$  in $C_n$ 
	is \textit{roughly} $|i-j|$.  
	More precisely there are constants $E > 1, F >0$ such that 
	\begin{equation} \label{linear growth}
		\frac{1}{E} |i-j| {\rm ent}\, \varphi - F 
		\leq  d(\varphi^{i}(\gamma), \varphi^{j}(\gamma))  
		\leq E |i-j| {\rm ent} \, \varphi + F.
	\end{equation} 
	To see this we identify the curve $\varphi^k(\gamma_*)$ with the obvious
	curve in the fibre $\Sigma \times \{k\,{\rm ent}\, \varphi \}$ to obtain a family 
	of curves in the model manifold  $M_n$  all of which have the same length $L$ say.
	The distance between $\varphi^i(\gamma_*)$ and $\varphi^j(\gamma_*)$
	in the metric on the model is \textit{exactly} $|i-j|\,{\rm ent} \, \varphi$. 
	Push forward using the $E$-bilipschitz homeomorphism that Minsky constructs
	(hence the factors $E^{\pm 1}$)
 	to obtain  a pair of curves homotopic to the  geodesics 
 	$\varphi^i(\gamma)$ and $\varphi^j(\gamma)$ respectively.
 	Observe that the  lengths of these curves are bounded by $EL$
 	so that they are at bounded distance from the closed geodesics
 	by the Morse Lemma
 	(hence the terms $\pm F$).   
\end{enumerate}
\medskip 

An important notion, due to McMullen, is that of   \textit{depth} in the convex core
which,  for a set of points,  is defined to be the minimum distance to the boundary of the convex core.
The inequality  (\ref{linear growth}) can be used 
to prove an estimate for the depth of the geodesic  $\varphi^{k}(\gamma)$:

\begin{lemma}
With the notation above,
there exists $F_1>0$, which does not depend on $n$,
such that for $-n \leq k \leq n$
\begin{equation}\label{depth growth}
		\frac{1}{E} ( n-|k|)\,{\rm ent}\, \varphi - F_1
		\leq 
	d(\partial C_n, \varphi^{k}(\gamma))   
	\leq 
	E( n-|k|)\,{\rm ent}\, \varphi +  F_1
\end{equation}
\end{lemma}
Brock-Bromberg  \cite{BB} give a proof of an analogous  inequality
without the hypothesis of bounded geometry.
As  we will use (\ref{depth growth})  in an essential way in the proof of Theorem \ref{Thm:Brock}, 
we give a short proof .

\begin{proof}
The inequality follows from (\ref{linear growth}) 
provided the distances 
$d(\partial^+ C_n, \varphi^{- n}(\gamma))$ 
and $d(\partial^- C_n, \varphi^{ n}(\gamma))$ are uniformy bounded.

The geodesic $\varphi^{-n}(\gamma)$ in $Q_n$ represents the curve 
$\varphi^{-n}(\gamma_*)$ on the surface at infinity $\varphi^{n}(X)$.
The length of $\varphi^{-n}(\gamma_*)$ on $\varphi^{n}(X)$
is equal to the length of $\gamma_*$  on $X$ 
and so is  less than  $L$.
By Epstein-Marden-Markovic and Bridgeman the nearest point retraction 
to $\partial^+ C_n$ is $5$-bilipschitz and applying this to $\gamma_*$
we obtain $\gamma_{**} \subset \partial^+ C_n$ of length at most $5L$.
Now, since the injectivity radius is bounded below by $\epsilon_0$,
 the Morse Lemma tells us that
$\gamma_{**}$
stays within an $R$-regular neighborhood of
the closed geodesic in its homotopy class namely $\varphi^{-n}(\gamma)$
so that,
\begin{equation*} 
	d(\partial^+ C_n, \varphi^{-n}(\gamma)) 
	\leq d(\gamma_{**}, \varphi^{-n}(\gamma))  \leq R.  
\end{equation*} 
To prove the lower bound one chooses 
 a piecewise geodesic path  joining $\varphi^{-n}(\gamma)$
 to $\varphi^{-k}(\gamma)$ and passing via $\partial^+ C_n $.
 The length of this path gives an upper bound for 
 $d(\varphi^{-n}(\gamma), \varphi^{-k}(\gamma)) $
\begin{equation*} 
	d(\varphi^{-n}(\gamma), \varphi^{-k}(\gamma))   
	\leq d(\partial^+ C_n, \varphi^{-n}(\gamma)) + 
	\text{diam} \, \partial^+ C_n + d(\partial^+ C_n, \varphi^{-k}(\gamma))  + L 
\end{equation*} 
so that 
\begin{align*}
	d(\partial^+ C_n, \varphi^{-k}(\gamma)) 
	& \geq d(\varphi^{-n}(\gamma), \varphi^{-k}(\gamma))- R  
		- \text{diam} \, \partial^+ C_n - L \\
	& \geq \left( \frac{1}{E} |n-k| {\rm ent}\, \varphi - F \right)
		- R  - 5 \,\text{diam} X  - L
\end{align*}
using (\ref{linear growth}) and the fact that  $X$  and  $\partial^+ C_n$  are 5-bilipschitz.
Thus the distance from  $\varphi^{-k}(\gamma)$  to  $\partial^+ C_n$
is roughly  $|n-k|$.
Replacing  $\varphi^{-n}(\gamma)$
by  $\varphi^{n}(\gamma)$  and applying the same reasoning, 
one obtains an analogous lower bound for the distance from 
$\varphi^{-k}(\gamma)$  to  $\partial^- C_n$
in terms of  $|n +k|$.
Combining the two bounds yields the required lower bound for 
the depth of  $\varphi^{k}(\gamma) $ 
in terms of  $\min{(|n-k|,|n+k|)} = n - |k|$.

The upper bound is proved in the same way.
\end{proof} 

We will now apply this to prove:

\begin{theorem}[see Brock \cite{Brock}]\label{Thm:Brock} 
Suppose  $\varSigma$  is compact, 
then 
\begin{equation}\label{Eq:vol N vs vol C}  
	|{\rm vol} \, C(\varphi^{-n} X, \varphi^n X) - 2n \, {\rm vol} \, N_{\varphi}  |
\end{equation} 
is uniformly bounded.
\end{theorem} 

Our strategy is, following McMullen and Brock-Bromberg,
to  decompose the  convex core 
into a \textit{deep part} and a \textit{shallow part}.
We first show that the shallow part is ``negligible" then, 
by geometric inflexibility, 
we see that the deep part is ``almost isometric" to a large chunk of
the infinite cyclic cover  $\widetilde{N}_{\varphi}$
which we can explicitly describe.
Consequently  the volume of the deep part grows like  $2n \, {\rm vol} \, N_{\varphi}$.  

\begin{proof} 
Recall that 
the  existence of Minsky's model (see preceding paragraph)
for the convex core of  $Q_n = QF(\varphi^{-n}X, \varphi^{n}X)$
guarantees  a universal
lower bound for injectivity radii  of the family  $\{Q_n\}$.
This will simplify the argument that we present below.

Let  $d>0$  and define the \textit{$d$-deep part of  $Q_n$} to be,  
\begin{equation*} 
	D_n(d) := \{ x \in C_n \, ; \, d(x, \partial C_n) \geq  d \}.
\end{equation*}
Note that this is a proper subset of the convex core  $C_n \subset Q_n$.
Moreover,  
it follows from (\ref{depth growth})  that,  
for fixed  $d$  and  $n$  sufficiently large, 
$D_n(d) $ is non empty and that its width grows linearly in  $n$
(by \textit{width} we mean that the minimum distance 
connected components of  $Q_n \setminus D_n(d)$).
Finally, 
we define the \textit{shallow part of $C_n$} to be  the complement 
of the deep part, that is, $C_n - D_n(d)$.  

Let  $X \in \mathcal{T}$  be a point on the axis of  $\varphi$.
A slight modification of the proof of 
Theorem 8.3 of Brock and Bromberg \cite{BB}
yields:
given  $d$  sufficiently large,
there are constants  $N, K_1, K_2>0$  such that for all  $n > N(d)$
there is a diffeomorphism  $g_n : D_n(d) \to \widetilde{N}_{\varphi}$  
with bilipschitz distortion at a point  $x \in D_n(d)$  less than
\begin{equation}\label{epsilon bound}
	1 + \exp(-K_1 d(\partial C_n,x) + K_2).
\end{equation}
and where the constants $K_1,K_2$ depend on $\epsilon_0$,
that is the lower bound on the injectivity radii of the $Q_n$,
 and $\chi(\varSigma)$.
We have given a simplified statement of a more general result
 they obtain  because we are working in a geometrically bounded context.
In order to  prove  (\ref{Eq:vol N vs vol C})
we must obtain a description of  $E_n := g_n(D_n(d))$  and, 
in particular, 
estimate the number of translates of a well-chosen fundamental domain
that are contained in  $E_n$.
We begin by estimating how many copies of a given (short, simple) 
closed curve are contained in  $E_n$.
To facilitate the exposition, 
we will define
\begin{equation*} 
	\epsilon(H):= \exp(-K_1 d(\partial C_n, H)  + K_2).
\end{equation*} 
where  $H$  is a subset of  $C_n$.

\vspace{.2in}
\begin{figure}[h]
\begin{center}
\includegraphics[scale=.4]{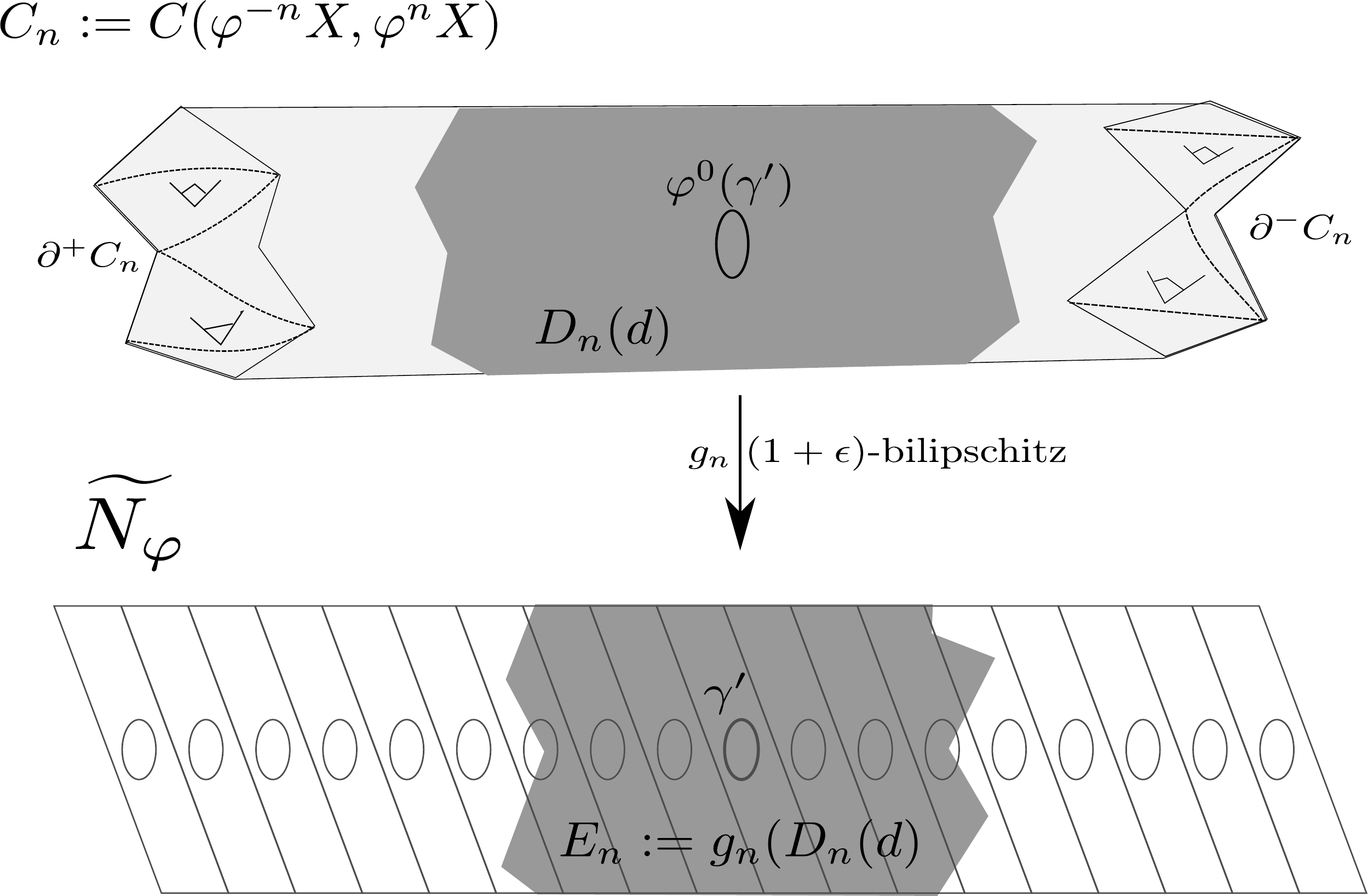} 
\end{center}
\caption{The map from the deep part to the infinite cyclic cover
which is shown tessellated by $\phi$-translates of a fundamental domain
 (depicted as parallograms)
each containing a closed geodesic $\varphi^k(\gamma')$.
}
\end{figure}

Choose a homotopy class  $\gamma_*$  of a simple closed curve on  $X$, 
and denote the  geodesic representative of  $\varphi^k (\gamma)_*$  in  $C_n$  
by  $\varphi^k(\gamma)$.
In particular, 
there is a collection of  $2n+1$  closed geodesics in  $C_n$,  
\begin{equation*} 
	\Gamma_n := \{ \varphi^k(\gamma) \, : \, -n \leq k \leq n \}.   
\end{equation*}  
By Minsky \cite{Minsky} the  lengths of the geodesics in  $\Gamma_n$  
are uniformly bounded (i.e. not depending on $n$) from above by 
some $L>0$.

Consider the subset of geodesics belonging to  $\Gamma_n$ 
that are not contained in the deep part $D_n(d)$, 
or equivalently,  
the values of  $k$  such that  $d(\partial C_n, \varphi^{k}(\gamma)) < d$.
By  (\ref{depth growth}) one has the inequality
\begin{equation} 
	\frac{1}{E} ( n-|k|)\,{\rm ent}\, \varphi - F_1   
	\leq d(\partial C_n, \varphi^{k}(\gamma)) < d 
\end{equation} 
for  $F_1>0$  and  $E>1$  which do not depend on $n$.
One can explicitly compute  $A_1 > 0$  depending on  the depth
$d$, 
$\varphi$  and $X$    (but not on  $n$)
such that the number of values of $k$  which satisfy this inequality,
hence the number of curves in  $\Gamma_n$  not contained in  $D_n(d)$,
is less than $A_1$.

Thus the deep part  $D_n(d)$  contains at least  $2n+1-A_1$ 
members of  $\Gamma_n$  and the image  $E_n$   
contains the same number of  $g_n(\varphi^{k}(\gamma))$.
Our objective is to ``promote" each of these latter  curves to a fundamental domain
for the action of  $\varphi$ by translation   on  $\widetilde{N}_{\varphi}$
contained in  $E_n$.  
The curve  $g_n(\gamma)$  is homotopic to a closed  geodesic 
$\gamma' \subset \widetilde{N}_\varphi$.
Since the length of  $\gamma'$  is bounded below ($N_\varphi$ is compact)
and the length of  $g_n(\gamma)$  uniformly bounded above,
by the Morse Lemma stated in point (1) above, 
the curve  $g_n(\gamma)$  is contained in an $R$-neighborhood of  $\gamma'$.
It follows that there exists  $A_2 \geq A_1$  such that if  $n -|k| > A_2$
then $\varphi^k(\gamma)$  and  $g_n^{-1} (\varphi^k(\gamma') )$  are 
contained in  $D_n(d)$.
 
Now define
\begin{equation} 
	\Delta
	:= \{  x, d ( x, \gamma') \leq d ( x, \varphi^{\pm 1} (\gamma') \} \subset \widetilde{N}_\varphi. 
\end{equation} 
Clearly, 
the interior of  $\Delta$  is a fundamental domain 
for the action of  $\varphi$  on  $\widetilde{N}_\varphi$  and 
the diameter of  $\Delta$  is bounded since  $N_\varphi$  is compact.  
Moreover,
\begin{align*} 
	& d(\partial C_n , \varphi^k(\gamma)) \\ 
	& \quad  \leq  d(\partial C_n , g_n^{-1} (\varphi^k(\Delta) \cap E_n))  
		+ \text{diam}\, g_n^{-1} (\varphi^k(\Delta) \cap E_n) + 
			 d(  g_n^{-1} (\varphi^k(\gamma') ), \varphi^k(\gamma)).
\end{align*}  
So that, 
if  $g_n^{-1}(\varphi^k(\gamma)) \subset D_n(d)$  then 
\begin{align*}
	& d(\partial C_n , g_n^{-1} (\varphi^k(\Delta) \cap E_n)) \\ 
	& \qquad \quad \geq  d(\partial C_n, \varphi^k(\gamma)) 
		- (1+\epsilon) \text{diam} \Delta
		- d(  g_n^{-1} (\varphi^k(\gamma') ), \varphi^k(\gamma))  \\
	& \qquad \quad \geq  d(\partial C_n, \varphi^k(\gamma))
		- (1+\epsilon) \,(\text{diam} \Delta +  R).  
\end{align*}
where $\epsilon = \epsilon(D_n(d))$. 
Consequently there exists  $A_3 \geq A_2$  such that if
$n -|k| \geq A_3$  then  $g_n^{-1}(\varphi^k(\Delta)) \subset D_n(d)$,
that is,
\begin{equation}\label{subsets}
	\bigcup_{|k| < n - A_3}  g_n^{-1}(\varphi^k(\Delta)) \subset D_n(d).
\end{equation}
From the above estimate one also obtains 
an explicit ``linear" lower bound for the depth 
by using the information contained in  Minsky's model,  
$g_n^{-1} (\varphi^k(\Delta) \cap E_n)$,   
\begin{equation}\label{lingro}
	d(\partial C_n , g_n^{-1} (\varphi^k(\Delta) \cap E_n))  
	\geq  \frac{1}{E} \, 	 (n - |k|)  {\rm ent} \, \varphi - F_1.
 \end{equation}
The inclusion (\ref{subsets}) yields a lower bound  for the volume as follows
\begin{align*}
	{\rm vol} \,  D_n(d) 
	&  \geq  \sum_{|k| \leq  n - A_3}  {\rm vol}\,  g_n^{-1}(\varphi^k(\Delta)) \\
	&  \geq  \sum_{|k| \leq n - A_3}  (1+ \epsilon(g_n^{-1}(\varphi^k(\Delta)))^{-3} {\rm vol} \,  \varphi^k(\Delta) \\
	&  \geq  \sum_{|k| \leq n - A_3}  (1 - 4\epsilon(g_n^{-1}(\varphi^k(\Delta))) {\rm vol} \,  \varphi^k(\Delta) \\
	&  \geq   2(n - A_3) {\rm vol} \,  \Delta 
		- 4({\rm vol} \,  \Delta) \sum_{|k| \leq  n - A_3}  \epsilon(g_n^{-1}(\varphi^k(\Delta)))
\end{align*}
The first term is the sum of  $2n  {\rm vol} \,  \Delta$  plus a quantity which 
does not depend on  $n$  and,
by (\ref{lingro}),
the second is bounded above by the sum of a geometric series.  

In order to bound the volume from above
we cover  the convex core by two sets.
Define,
$$
\begin{array}{rcl}
	 V_n &:=& \bigcup_{|k| \leq  n - A_3}  g_n^{-1}(\varphi^k(\overline{\Delta}))\\	 
	 &&\\
	 S_n  &: =& g_n^{-1}(\varphi^{-n + A_3} (\Delta)) 
\cup g_n^{-1}(\varphi^{n - A_3} (\Delta)).
\end{array}
$$
so that one obtains the convex core as the union
$V_n \cup (S_n \cup (C_n \setminus V_n))$.
Observe  that $\Delta$ separates $N_\varphi$
so $S_n$ separates $C_n$
and the  complement of
consists of 3 components;
one containing $\partial^+ C_n$,
another $\partial^- C_n$ 
and the remaining component 
$g^{-1}(\varphi^k)$ for  $k  <  |n| - A_3$.
Clearly,
the  depth of a point contained in the 
the components that contain the boundary $\partial C_n$
is bounded from above by the maximum depth of a point in
$S_n$.
Let us give an explicit upper bound using (\ref{depth growth}).
For $x \in g^{-1}(\varphi^k (\Delta) )$ one has 
$$d(\partial^- C_n , x) 
\leq d(\partial^- C_n , \varphi^k(\gamma)) + d( \varphi^k(\gamma),x ) 
\leq 	E( n-|k|)\,{\rm ent}\, \varphi  + F_1 + R + (1+\epsilon) \text{diam}\, \Delta.
 $$
 So that, setting  $k= A_3$ in the expression on the right,
 we obtain   the required  uniform bound, w
(it is easy to see that  $g^{-1}(\varphi^k (\Delta) )$
satisfies the same upper bound).
It follows that  the complement in the convex core of 
$V_n$ is contained in a $T$-regular neighborhood of the boundary 
$\partial C_n$
that is,
\begin{equation*} 
	C_n = N_T(\partial C_n)\, \cup  V_n.
\end{equation*} 
From this we obtain the upper bound for volume
\begin{equation*} 
	{\rm vol} \,  C_n \leq  {\rm vol} \,  N_T(\partial C_n) + {\rm vol} \,   
	V_n.
\end{equation*} 
The diameter of each of the  components of  $ N_T(\partial C_n)$
is uniformly bounded  bounded  (by $T + 5 \,\text{diam} X$)
  so the first term is uniformly bounded in  $n$  and, 
by a similar calculation to the above for the lower bound,
the second term is seen to be bounded from above by 
$2n  {\rm vol} \,  \Delta$  plus a constant.
\end{proof}



\bigskip

\begin{thebibliography}{99} 

\bibitem{Agol} 
I.~Agol, 
Ideal triangulations of pseudo-Anosov mapping tori, 
Contemporary Math., 560 (2011), 1-17.  

\bibitem{BersI} 
L.~Bers, 
Simultaneous uniformization, 
Bull. Amer. Math. Soc., 66 (1960), 94-97.  

\bibitem{BersII} 
L.~Bers, 
An extremal problem for quasiconformal mappings and 
a theorem by Thurston, 
Acta. Math., 141 (1978), 73-98.  



\bibitem{Bridgeman} 
M.~Bridgeman, 
Average bending of convex pleated planes in hyperbolic three-space, 
Invent. Math., 132 (1998), 381-391.  

\bibitem{Brock} 
J.~Brock, 
Weil-Petersson translation distance and volumes of mapping tori, 
Communications in Analysis and Geometry, 11 (2003),  
987--999. 

\bibitem{BB} 
J.~Brock and K.~Bromberg, 
Geoometric inflexibility and 3-manifolds that fiber over the circle, 
J. Topol., 4 (2011), 1-38.  

\bibitem{CM} 
C.~Cao and R.~Meyerhoff, 
The orientable cusped hyperbolic 3-manifolds of minimal volume, 
Invent. Math., 146 (2001),  
451--478. 

\bibitem{Epstein} 
C.~Epstein, 
Envelopes of horospheres and Weingarten surfaces 
in hyperbolic 3-spaces, 
preprint, Princeton Univ., (1984).  

\bibitem{EMM} 
D. B. A. ~Epstein, A. ~Marden and V. ~Markovic, 
Quasiconformal homeomorphisms and the convex hull boundary, 
Ann. of Math., 59 (2004),  305-336.

\bibitem{FLM} 
B.~Farb, C. Leininger and D. Margalit, 
Small dilatation pseudo-Anosovs and 3-manifolds, 
Adv. Math., 228 (2011), 1466-1502.  

\bibitem{GL} 
F.~Gardiner and N.~Lakic, 
Quasiconformal Teichm\"uller Theory, 
Mathematical Surveys and Monographs, Volume 76, 
Amer. Math. Soc., (2000).  

\bibitem{KKT} 
E.~Kin, S.~Kojima and M.~Takasawa, 
Entropy versus volume for pseudo-Anosovs, 
Experimental Math., 18 (2009), 397-407.  

\bibitem{Kojima} 
S.~Kojima, 
Entropy, Weil-Petersson translation distance and Gromov norm 
for surface automorphism, 
Proc. Amer. Math. Soc., 140 (2012), 3993-4002. 

\bibitem{KS}
K.~Krasnov and J-M. Schlenker, 
On the renomalized volume of hyperbolic 3-manifolds, 
Comm. Math. Phy. 279 (2008), 637-668.  

\bibitem{McMullen} 
C.~McMullen, 
Renormalization and 3-manifolds which fiber over the circle, 
Ann. Math. Study 142 (1996).  

\bibitem{Minsky}
Y.~Minsky, 
Bounded geometry for Kleinian groups,
Invent. Math., 146 (2001),
143-192,
  
\bibitem{Nehari} 
Z.~Nehari, 
The Schwarzian derivative and schlicht functions, 
Bull. Amer. Math. Soc., 55 (1949), 545-551.  

\bibitem{Schlenker} 
J-M,~Schlenker, 
The renormalized volume and the volume of the convex core 
of quasifuchsian manifolds, 
Math. Res. Lett., 20 (2013), 773-786.  

\bibitem{Penner} 
R.~Penner, 
Bounds on least dilatations, 
Proc. Amer. Math. Soc., 113 (1991), 443-450.  

\bibitem{Rafi}
K. ~Rafi
A characterization of short curves of a Teichmeuller  geodesic,
Geometry \&  Topology Volume 9 (2005) 179???202

\bibitem{Tsai} 
C. Y.~Tsai, 
The asymptotic behavior of least pseudo-Anosov dilatations, 
Geometry and Topology, 13 (2009), 2253--2278.  

\bibitem{ThurstonI}
W.~Thurston, 
The geometry and topology of $3$-manifolds, 
Lecture Notes, Princeton University (1979). 

\bibitem{ThurstonII} 
W.~Thurston, 
On the geometry and dynamics of diffeomorphisms of surfaces, 
Bulletin of Amer. Math. Society., 
19 (1988), 417--431.

\bibitem{ThurstonIII} 
W.~Thurston, 
Hyperbolic structures on 3-manifolds II: Surface groups and 
3-manifolds which fiber over the circle, 
preprint.  

\end{thebibliography}
\end{document}